\documentclass[11pt,reqno]{amsart}
\usepackage{graphicx}
\usepackage{verbatim}
\usepackage{textcomp}
\usepackage{amssymb}
\usepackage{cite}
\usepackage{amsmath}
\usepackage{latexsym}
\usepackage{amscd}
\usepackage{amsthm}
\usepackage{mathrsfs}
\usepackage{xypic}
\usepackage{bm}
\usepackage{url}

\newtheorem{thm}{Theorem}

\newtheorem{lem}[thm]{Lemma}
\newtheorem{prop}[thm]{Proposition}

\theoremstyle{definition}

\newtheorem*{ack}{Acknowledgment}
\def\R{\mathbb R}

\def\H{\mathbb H}
\def\SS{\mathbb S}

\def\ra{\rightarrow}
\def\pt{\partial}

\begin{document}
\title{A geometric inequality on hypersurface in hyperbolic space}
\author{Haizhong Li}
\address{Department of mathematical sciences, and Mathematical Sciences
Center, Tsinghua University, 100084, Beijing, P. R. China}
\email{hli@math.tsinghua.edu.cn}
\author{Yong Wei}
\address{Department of mathematical sciences, Tsinghua University, 100084, Beijing, P. R. China}
\email{wei-y09@mails.tsinghua.edu.cn}
\author{Changwei Xiong}
\address{Department of mathematical sciences, Tsinghua University, 100084, Beijing, P. R. China}
\email{xiongcw10@mails.tsinghua.edu.cn}
\thanks{The research of the authors was supported by NSFC No. 11271214.}

\maketitle

\begin{abstract}
In this paper, we use the inverse curvature flow to prove a sharp geometric inequality on star-shaped and two-convex hypersurface in hyperbolic space.
\end{abstract}

\section{Introduction}

The classical Alexandrov-Fenchel inequalities for closed convex hypersurface $\Sigma\subset\R^n$ state that
\begin{align}\label{eq1-1}
    \int_{\Sigma}\sigma_m(\kappa)d\mu\geq C_{n,m}(\int_{\Sigma}\sigma_{m-1}(\kappa)d\mu)^{\frac{n-m-1}{n-m}},\quad 1\leq m\leq n-1
\end{align}
where $\sigma_m(\kappa)$ is the $m$-th elementary symmetric polynomial of the principal curvatures $\kappa=(\kappa_1,\cdots,\kappa_{n-1})$  of $\Sigma$  and $C_{n,m}=\frac{\sigma_m(1,\cdots,1)}{\sigma_{m-1}(1,\cdots,1)}$ is a constant. When $m=0$, \eqref{eq1-1} is interpreted as the classical isoperimetric inequality
\begin{equation}\label{isop}
    |\Sigma|^{\frac 1{n-1}}\geq \bar{C}_n Vol(\Omega)^{\frac 1n},
\end{equation}
which holds on all bounded domain $\Omega\subset\R^n$ with boundary $\Sigma=\pt\Omega$. Here $|\Sigma|$ is the area $\Sigma$ and $\bar{C}_n$ is a constant depending only on dimension $n$. Inequality \eqref{eq1-1} was generalized to star-shaped and $m$-convex hypersurface $\Sigma\subset\R^n$ by Guan and Li \cite{GL} using the inverse curvature flow recently, where $m$-convex means that the principal curvature of $\Sigma$ lies in the Garding's cone
\begin{equation*}
    \Gamma_m=\{\kappa\in\R^{n-1}|\sigma_i(\kappa)>0,i=1,\cdots,m\}.
\end{equation*}
Recently, Huisken \cite{Hu1} showed that in the case $m=1$, the assumption {\it star-shaped} can be replaced by {\it outward-minimizing}.

In this paper, we consider the hyperbolic space $\H^n=\R^+\times\SS^{n-1}$ endowed with the metric
\begin{equation*}
    \bar{g}=dr^2+\sinh^2rg_{\SS^{n-1}},
\end{equation*}
where $g_{\SS^{n-1}}$ is the standard round metric on the unit sphere $\SS^{n-1}$. It's a natural question to establish some analogue inequalities of \eqref{eq1-1} for closed hypersurface in $\H^n$. In the case of $m=1$, $\sigma_1=\sigma_1(\kappa)$ is just the mean curvature $H$ of $\Sigma$. Gallego and Solanes \cite{GS} have obtained a generalization of \eqref{eq1-1} to convex hypersurface in hyperbolic space using integral geometric methods, however, their result does not seem to be sharp. Denoting $\lambda(r)=\sinh r$, then $\lambda'(r)=\cosh r$.  Recently,  Brendle,Hung and Wang \cite{BHW} proved the following inequality for star-shaped and mean convex (i.e.,$H>0$) hypersurface $\Sigma\subset\H^n$:
\begin{equation}\label{eq1-3}
    \int_{\Sigma}(\lambda'H-(n-1)\langle\bar{\nabla}\lambda',\nu\rangle)d\mu\geq(n-1)\omega_{n-1}^{\frac 1{n-1}}{|\Sigma|}^{\frac{n-2}{n-1}}
\end{equation}
where $|\Sigma|$ is the area of $\Sigma$ and $\omega_{n-1}$ is the area of the unit sphere $\SS^{n-1}\subset\R^n$. de Lima and Girao \cite{LimaG} also proved the following related inequality independently.
\begin{equation}\label{eq1-2}
    \int_{\Sigma}\lambda'Hd\mu\geq (n-1)\omega_{n-1}\left((\frac {|\Sigma|}{\omega_{n-1}})^{\frac{n-2}{n-1}}+(\frac {|\Sigma|}{\omega_{n-1}})^{\frac{n}{n-1}}\right),
\end{equation}
Both inequalities \eqref{eq1-3} and \eqref{eq1-2} are sharp in the sense of that equality holds if and only if $\Sigma$ is a geodesic sphere centered at the origin. Here, we say a closed hypersurface $\Sigma\subset\H^n$ is star-shaped if the unit outward normal $\nu$ satisfies $\langle\nu,\pt_r\rangle>0$ everywhere on $\Sigma$, which is also equivalent to that $\Sigma$ can be parametrized by a graph
\begin{equation*}
    \Sigma=\{(r(\theta),\theta)|\theta\in\SS^{n-1}\}
\end{equation*}
for some smooth function $r$ on $\SS^{n-1}$.  We note that inequalities \eqref{eq1-3} and \eqref{eq1-2} have some applications in general relativity, see \cite{BHW,LimaG,W}.

In this paper, we consider the case $m=2$. We prove the following sharp inequality for star-shaped and two-convex hypersurface $\Sigma\subset\H^n$, where {\it two-convex} means that the principal curvature lies in the Garding's cone $\Gamma_2$ everywhere on $\Sigma$.
\begin{thm}\label{mainthm}
If $\Sigma\subset\H^n$ is a star-shaped and two-convex hypersurface, then
\begin{align}\label{maineq}
    \int_{\Sigma}\sigma_2d\mu\geq&\frac{(n-1)(n-2)}2~\left(\omega_{n-1}^{\frac{2}{n-1}}{|\Sigma|}^{\frac{n-3}{n-1}}+|\Sigma|\right),
\end{align}
where $\omega_{n-1}$ is the area of the unit sphere $\SS^{n-1}\subset\R^n$ and $|\Sigma|$ is the area of $\Sigma$. The equality holds if and only if $\Sigma$ is a geodesic sphere.
\end{thm}

Note that there exists at least one elliptic point on a closed, connected hypersurface $\Sigma$ in hyperbolic space $\H^n$. Proposition 3.2 in \cite{BC} shows that if $\sigma_2$ is positive, then $\sigma_1$ is automatically positive. So our assumption {\it two-convex} can also be replaced by $\sigma_2>0$ on $\Sigma$.

The proof of Theorem \ref{mainthm} follows a similar argument as in \cite{BHW,GL,LimaG}. We evolve $\Sigma$ by a special case of the inverse curvature flow in \cite{Ge}, and consider the following  quantity defined by
\begin{align*}
    Q(t)=&{|\Sigma|}^{-\frac{n-3}{n-1}}\left(\int_{\Sigma}\sigma_2d\mu-\frac{(n-1)(n-2)}2|\Sigma|\right).
\end{align*}
We show that $Q(t)$ is monotone decreasing under the flow. Then we use the convergence result of the flow proved by Gerhardt to estimate a lower bound of the limit of $Q(t)$:
\begin{align*}
    \liminf_{t\ra\infty}Q(t)\geq \frac{(n-1)(n-2)}2~\omega_{n-1}^{\frac 2{n-1}}.
\end{align*}
In order to estimate this $\liminf$, we also use a sharp version Sobolev inequality on $\SS^{n-1}$ due to Beckner \cite{Be} as in \cite{BHW}. Finally Theorem \ref{mainthm} follows easily from the monotonicity and the lower bound of $\liminf_{t\ra\infty}Q(t)$.

\begin{ack}
The authors would like to thank Professor Ben Andrews for helpful discussions and his interest in this work. The second author also thank Frederico Gir\~{a}o for communications regarding Appendix A of \cite{LimaG}.
\end{ack}

\section{Preliminaries}

Let $\Sigma\subset\H^n$ be a closed hypersurface with unit outward normal $\nu$. The second fundamental form $h$ of $\Sigma$ is defined by
\begin{equation*}
    h(X,Y)=\langle\bar{\nabla}_X\nu,Y\rangle
\end{equation*}
for any two tangent fields $X,Y$. The principal curvature $\kappa=(\kappa_1,\cdots,\kappa_{n-1})$ are the eigenvalues of $h$ with respect to the induced metric $g$ on $\Sigma$. For $1\leq m\leq n-1$, the $m$-th elementary symmetric polynomial of $\kappa$ is defined as
\begin{equation*}
   \sigma_m(\kappa)=\sum_{i_1<i_2\cdots<i_m}\kappa_{i_1}\cdots\kappa_{i_m},
\end{equation*}
which can also be viewed as function of the second fundamental form $h_i^j=g^{jk}h_{ki}$.  In the sequel, we will simply write $\sigma_m$ for $\sigma_m(\kappa)$. We first collect the following basic facts on $\sigma_m$ (see, e.g,\cite{HLP,Hu2,Re}):

\begin{lem}
Denote $(T_{m-1})_j^i=\frac{\pt\sigma_m}{\pt h_i^j}$  and $(h^2)_i^j=g^{jl}g^{pk}h_{kl}h_{ip}$. We have
\begin{align}
    &\sum_{i,j}(T_{m-1})_j^ih_{i}^j=m\sigma_m,\\
    &\sum_{i,j}(T_{m-1})_j^i\delta_i^j=(n-m)\sigma_{m-1}\label{eq2-0}\\
    &\sum_{i,j}(T_{m-1})_j^i(h^2)_i^j=\sigma_1\sigma_m-(m+1)\sigma_{m+1}\label{eq2-1}
\end{align}
Moreover, if $\kappa\in\Gamma_m^+$, we have the following Newton-MacLaurin inequalities
\begin{align}
    &\frac{\sigma_{m-1}\sigma_{m+1}}{\sigma_m^2}\leq \frac{m(n-m-1)}{(m+1)(n-m)}\label{eq2-2}\\
    &\frac{\sigma_1\sigma_{m-1}}{\sigma_m}\geq \frac{m(n-1)}{n-m},\label{eq2-3}
  \end{align}
and the equalities hold in \eqref{eq2-2},\eqref{eq2-3} at a given point if and only if $\Sigma$ is umbilical there.
\end{lem}

We now evolve $\Sigma\subset\H^n$ by the following evolution equation
\begin{equation}\label{evolu}
    \pt_tX=F\nu,
\end{equation}
where $\nu$ is the unit outward normal to $\Sigma_t=X(t,\cdot)$ and $F$ is the speed function which may depend on the position vector, principal curvatures and time. Let $g_{ij}$ be the induced metric and $d\mu_t$ be its area element on $\Sigma_t$. We have the following evolution equations.

\begin{prop}
Under the flow \eqref{evolu}, we have:
\begin{align}
    &\pt_tg_{ij}=2Fh_{ij}\nonumber\\
    &\pt_t\nu=-\nabla F,\nonumber\\
    &\pt_th_i^j=-\nabla^j\nabla_iF-F(h^2)_i^j+F\delta_i^j,\nonumber\\
    &\pt_td\mu=F\sigma_1d\mu,\label{eq2-4}\\
    &\pt_t\sigma_m=-\nabla^i((T_{m-1})_i^j\nabla_jF)-F(\sigma_1\sigma_m-(m+1)\sigma_{m+1})\nonumber\\
    &\qquad\qquad+(n-m)F\sigma_{m-1},    \label{eq2-5}
\end{align}
\end{prop}
\proof
The first four equations follow from direct computations like in \cite{Hu}. Now we calculate the evolution of $\sigma_m$ (cf. \cite{GL})
\begin{align*}
    \pt_t\sigma_m=&\frac{\pt\sigma_m}{\pt h_i^j}\pt_t h_i^j\\
    =&-(T_{m-1})_j^i\nabla^j\nabla_iF-F(T_{m-1})_j^i(h^2)_i^j+F(T_{m-1})_j^i\delta_i^j\\
    =&-\nabla^j((T_{m-1})_j^i\nabla_iF)-F(\sigma_1\sigma_m-(m+1)\sigma_{m+1})\nonumber\\
    &\qquad\qquad+(n-m)F\sigma_{m-1},
\end{align*}
where in the last equality we used \eqref{eq2-0},\eqref{eq2-1} and the divergence free property of $(T_{m-1})_j^i$ (see \cite{Re}).
\endproof

\begin{prop}\label{prop2-3}
Under the flow \eqref{evolu}, we have
\begin{align*}
    \frac d{dt}\int_{\Sigma}\sigma_md\mu=&(m+1)\int_{\Sigma}F\sigma_{m+1}d\mu+(n-m)\int_{\Sigma}F\sigma_{m-1}d\mu.
\end{align*}
\end{prop}
\proof
This proposition follows directly from \eqref{eq2-4}, \eqref{eq2-5} and the divergence theorem.
\endproof

In \cite{Ge} Gerhardt studied general inverse curvature flow of star-shaped hypersurface in hyperbolic space. For our purpose, we will use a special case of their result for the following flow
\begin{equation}\label{evolu-1}
    \pt_tX=\frac{n-2}{2(n-1)}\frac{\sigma_1}{\sigma_2}\nu.
\end{equation}
\begin{thm}[Gerhardt\cite{Ge}]
If the initial hypersurface is star-shaped and strictly two-convex, then the solution for the flow \eqref{evolu-1} exists for all time $t>0$ and the flow hypersurfaces converge to infinity while maintaining star-shapedness and strictly two-convex. Moreover, the hypersurfaces become strictly convex exponentially fast and more and more totally umbilical in the sense of
\begin{equation*}
    |h_i^j-\delta_i^j|\leq Ce^{-\frac t{n-1}},\quad t>0,
\end{equation*}
i.e., the principal curvatures are uniformly bounded and converge exponentially fast to one.
\end{thm}

\section{Monotonicity}

We define the quantity
\begin{align*}
    Q(t)=&|\Sigma_t|^{-\frac{n-3}{n-1}}\left(\int_{\Sigma_t}\sigma_2d\mu-\frac{(n-1)(n-2)}2|\Sigma_t|\right),
\end{align*}
where $|\Sigma_t|$ is the area of $\Sigma_t$. In this section, we show that $Q(t)$ is monotone decreasing along the flow \eqref{evolu-1}.
\begin{prop}\label{mono}
Under the flow \eqref{evolu-1}, the quantity $Q(t)$ is monotone decreasing. Moreover, $\frac d{dt}Q(t)=0$ at some time $t$ if and only if the surface $\Sigma_t$ is totally umbilical.
\end{prop}
\proof
Under the flow \eqref{evolu-1}, Proposition \ref{prop2-3} and \eqref{eq2-4} imply that
\begin{align}
    \frac d{dt}\int_{\Sigma}\sigma_2d\mu=&\frac{3(n-2)}{2(n-1)}\int_{\Sigma}\frac{\sigma_1\sigma_3}{\sigma_2}d\mu+\frac{(n-2)^2}{2(n-1)}\int_{\Sigma}\frac{\sigma_1^2}{\sigma_2}d\mu\label{eq3-1}\\
    \frac d{dt}|\Sigma_t|=&\frac{(n-2)}{2(n-1)}\int_{\Sigma}\frac{\sigma_1^2}{\sigma_2}d\mu.\label{eq3-2}
\end{align}
Combining \eqref{eq3-1}, \eqref{eq3-2} and \eqref{eq2-2}, we have
\begin{align}\label{eq3-3}
    \frac d{dt}\left(\int_{\Sigma}\sigma_2d\mu-(n-2)|\Sigma_t|\right)\leq &\frac{n-3}{n-1}\int_{\Sigma}\sigma_2d\mu.
\end{align}
By applying the Newton-MacLaurin inequality \eqref{eq2-3} in \eqref{eq3-2}, we also have
\begin{align}\label{eq3-4}
    \frac d{dt}|\Sigma_t|\geq&~|\Sigma_t|.
\end{align}
Then combining \eqref{eq3-3} and \eqref{eq3-4} gives that
\begin{align}\label{eq3-5}
    \frac d{dt}\left(\int_{\Sigma}\sigma_2d\mu-\frac{(n-1)(n-2)}2|\Sigma_t|\right)\leq&\frac{n-3}{n-1}\left(\int_{\Sigma}\sigma_2d\mu-\frac{(n-1)(n-2)}2|\Sigma_t|\right).
\end{align}

From Proposition \ref{prop4-2} in the next section and \eqref{eq3-5}, we know that the quantity
\begin{equation*}
   \int_{\Sigma}\sigma_2d\mu-\frac{(n-1)(n-2)}2|\Sigma_t|
\end{equation*}
is nonnegative along the flow \eqref{evolu-1}. Then inequalities \eqref{eq3-4} and \eqref{eq3-5} implies that
\begin{align*}
    \frac d{dt}Q(t)\leq 0.
\end{align*}
If the equality holds, the inequalities \eqref{eq2-2} and \eqref{eq2-3} assume equalities everywhere on $\Sigma_t$. Then $\Sigma_t$ is totally umbilical.
\endproof

\section{The asymptotic behavior of monotone quantity}

In this section, we use the convergence result of the flow \eqref{evolu-1} proved in \cite{Ge} to estimate the lower bound of the limit of $Q(t)$. First we need the following sharp Sobolev inequality on $\SS^{n-1}$ (\cite{Be}).
\begin{lem}\label{lem4-1}
For every positive function $f$ on $\SS^{n-1}$, we have
\begin{align*}
     &\int_{\SS^{n-1}}f^{n-3}dvol_{\SS^{n-1}}+\frac{n-3}{n-1}\int_{\SS^{n-1}}f^{n-5}|\nabla f|^2dvol_{\SS^{n-1}}\\
     &\quad \geq \omega_{n-1}^{\frac 2{n-1}}(\int_{\SS^{n-1}}f^{n-1}dvol_{\SS^{n-1}})^{\frac{n-3}{n-1}}.
\end{align*}
Moreover, equality holds if and only if $f$ is a constant.
\end{lem}
\proof
From Theorem 4 in \cite{Be}, for any positive smooth function $w$ on $\SS^{n-1}$, we have the following inequalty
\begin{align*}
    &\frac 4{(n-1)(n-3)}\int_{\SS^{n-1}}|\nabla w|^2dvol_{\SS^{n-1}}+\int_{\SS^{n-1}}w^2dvol_{\SS^{n-1}}\\
    &\quad \geq \omega_{n-1}^{\frac 2{n-1}}(\int_{\SS^{n-1}}w^{\frac{2(n-1)}{n-3}}dvol_{\SS^{n-1}})^{\frac{n-3}{n-1}}.
\end{align*}
Moreover equality holds if and only if $w$ is constant. For any positive function $f$ on $\SS^{n-1}$, by letting $w=f^{\frac{n-3}2}$, we have
\begin{align*}
    &\int_{\SS^{n-1}}f^{n-3}dvol_{\SS^{n-1}}+\frac{n-3}{n-1}\int_{\SS^{n-1}}f^{n-5}|\nabla f|^2dvol_{\SS^{n-1}}\\
    &\quad \geq \omega_{n-1}^{\frac 2{n-1}}(\int_{\SS^{n-1}}f^{n-1}dvol_{\SS^{n-1}})^{\frac{n-3}{n-1}}
\end{align*}
and equality holds if and only if $f$ is a constant.
\endproof

\begin{prop}\label{prop4-2}
Under the flow \eqref{evolu-1}, we have
\begin{align}\label{asymp}
    \liminf_{t\ra\infty}Q(t)\geq \frac{(n-1)(n-2)}2~\omega_{n-1}^{\frac 2{n-1}}.
\end{align}
\end{prop}
\proof
Recall that star-shaped hypersurfaces can be written as graphs of function $r=r(t,\theta),\theta\in\SS^{n-1}$. Denote $\lambda(r)=\sinh(r)$, then $\lambda'(r)=\cosh(r)$. We next define a function $\varphi$ on $\SS^{n-1}$ by $\varphi(\theta)=\Phi(r(\theta))$, where $\Phi(r)$ is a positive function satisfying $\Phi'=1/\lambda$. Let $\theta=\{\theta^j\},j=1,\cdots,n-1$ be a coordinate system on $\SS^{n-1}$ and $\varphi_i,\varphi_{ij}$ be the covariant derivatives of $\varphi$ with respect to the metric $g_{\SS^{n-1}}$. Define
\begin{equation*}
    v=\sqrt{1+|\nabla\varphi|^2_{\SS^{n-1}}}.
\end{equation*}
From \cite{Ge}, we know that
\begin{equation}\label{eq4-1}
    \lambda=O(e^{\frac t{n-1}}),\qquad |\nabla\varphi|_{\SS^{n-1}}+|\nabla^2\varphi|_{\SS^{n-1}}=O(e^{-\frac t{n-1}})
\end{equation}
Since $\lambda'=\sqrt{1+\lambda^2}$, we have
\begin{equation}\label{eq4-3}
    \lambda'=\lambda(1+\frac 12\lambda^{-2}+O(e^{-\frac {4t}{n-1}}))
\end{equation}
From \eqref{eq4-1} we also have
\begin{equation}\label{eq4-2}
    \frac 1v=1-\frac 12|\nabla\varphi|^2_{\SS^{n-1}}+O(e^{-\frac {4t}{n-1}})
\end{equation}
In terms of $\varphi$, we can express the metric and second fundamental form of $\Sigma$ as following (see, e.g, \cite{BHW,Ding})
\begin{align*}
    g_{ij}=&\lambda^2(\sigma_{ij}+\varphi_i\varphi_j)\\
    h_{ij}=&\frac{\lambda'}{v\lambda}g_{ij}-\frac{\lambda}{v}\varphi_{ij},
\end{align*}
where $\sigma_{ij}=g_{\SS^{n-1}}(\pt_{\theta^i},\pt_{\theta^j})$. Denote $a_i=\sum_k\sigma^{ik}\varphi_{ki}$ and note that $\sum_ia_i=\Delta_{\SS^{n-1}}\varphi$. By \eqref{eq4-1}, the principal curvatures of $\Sigma_t$  has the following form
\begin{equation*}
    \kappa_i=\frac{\lambda'}{v\lambda}-\frac {a_i}{v\lambda}+O(e^{-\frac{4t}{n-1}}),\quad i=1,\cdots,n-1.
\end{equation*}
Then we have
\begin{align*}
    \sigma_2=&\sum_{i<j}\kappa_i\kappa_j\\
    =&\frac{(n-1)(n-2)}2(\frac{\lambda'}{v\lambda})^2-(n-2)\frac{\lambda'\Delta_{\SS^{n-1}}\varphi}{v^2\lambda^2}+O(e^{-\frac{4t}{n-1}}).
\end{align*}
By using \eqref{eq4-3} and \eqref{eq4-2},
\begin{align*}
    \sigma_2=&\frac{(n-1)(n-2)}2(1+\lambda^{-2}-|\nabla\varphi|^2_{\SS^{n-1}})\\
    &\quad -(n-2)\lambda^{-1}\Delta_{\SS^{n-1}}\varphi+O(e^{-\frac{4t}{n-1}}).
\end{align*}
On the other hand,
\begin{align*}
    \sqrt{\det g}=(\lambda^{n-3}+O(e^{\frac{(n-3)t}{n-1}}))\sqrt{\det g_{\SS^{n-1}}}.
\end{align*}
So we have
\begin{align*}
    &\int_{\Sigma_t}\sigma_2d\mu-\frac{(n-1)(n-2)}2|\Sigma_t|\\
    =&\int_{\SS^{n-1}}\lambda^{n-1}(\sigma_2-\frac{(n-1)(n-2)}2)dvol_{\SS^{n-1}}+O(e^{\frac{(n-5)t}{n-1}})\\
    =&\frac{(n-1)(n-2)}2\int_{\SS^{n-1}}(\lambda^{n-3}-\lambda^{n-1}|\nabla\varphi|^2_{\SS^{n-1}})dvol_{\SS^{n-1}}\\
    &\qquad -(n-2)\int_{\SS^{n-1}}\lambda^{n-2}\Delta_{\SS^{n-1}}\varphi dvol_{\SS^{n-1}}+O(e^{\frac{(n-5)t}{n-1}})\\
    =&\frac{(n-1)(n-2)}2\int_{\SS^{n-1}}(\lambda^{n-3}-\lambda^{n-1}|\nabla\varphi|^2_{\SS^{n-1}})dvol_{\SS^{n-1}}\\
    &\qquad+(n-2)^2\int_{\SS^{n-1}}\lambda^{n-3}\nabla\lambda\nabla\varphi dvol_{\SS^{n-1}}+O(e^{\frac{(n-5)t}{n-1}}).
\end{align*}
Since $\nabla\lambda=\lambda\lambda'\nabla\varphi$, by using \eqref{eq4-3}, we deduce that
\begin{align}
  &\int_{\Sigma_t}\sigma_2d\mu-\frac{(n-1)(n-2)}2|\Sigma_t|\nonumber\\
  =&\frac{(n-1)(n-2)}2\int_{\SS^{n-1}}(\lambda^{n-3}+\frac{n-3}{n-1}\lambda^{n-5}|\nabla\lambda|^2)dvol_{\SS^{n-1}}+O(e^{\frac{(n-5)t}{n-1}}).\label{eq4-4}
\end{align}
Moreover,
\begin{align}\label{eq4-5}
    |\Sigma_t|^{\frac{n-3}{n-1}}=&(\int_{\SS^{n-1}}\lambda^{n-1}dvol_{\SS^{n-1}})^{\frac{n-3}{n-1}}+O(e^{\frac{(n-5)t}{n-1}}).
\end{align}
Using Lemma \ref{lem4-1}, we can complete the proof of Proposition \ref{prop4-2} by combining \eqref{eq4-4} and \eqref{eq4-5}.
\endproof

We now complete the proof of Theorem \ref{mainthm}
\begin{proof}[Proof of Theorem \ref{mainthm}]
Since $Q(t)$ is monotone decreasing, we have
\begin{align*}
    Q(0)\geq\liminf_{t\ra\infty}Q(t)\geq \frac{(n-1)(n-2)}2~\omega_{n-1}^{\frac 2{n-1}}.
\end{align*}
This gives that the initial hypersurface $\Sigma$ satisfies
\begin{align*}
   \left(\int_{\Sigma}\sigma_2d\mu-\frac{(n-1)(n-2)}2|\Sigma|\right)\geq \frac{(n-1)(n-2)}2~\omega_{n-1}^{\frac 2{n-1}} |\Sigma|^{\frac{n-3}{n-1}},
\end{align*}
which is equivalent to the inequality \eqref{maineq} in Theorem \ref{mainthm}. Now we assume that equality holds in \eqref{maineq}, which implies that $Q(t)$ is a constant.  Then Proposition \ref{mono} implies $\Sigma_t$ is umbilical and therefore a geodesic sphere.
It is also easy to see that if $\Sigma$ is a geodesic sphere of radius $r$, then the area of $\Sigma$ is $|\Sigma|=\omega_{n-1}\sinh^{n-1}r$ and the integral of $\sigma_2$ is
\begin{align*}
   \int_{\Sigma}\sigma_2d\mu=&\frac{(n-1)(n-2)}2\omega_{n-1}\coth^2{r}\sinh^{n-1}r\\
   =&\frac{(n-1)(n-2)}2\omega_{n-1}\left(\sinh^{n-1}r+\sinh^{n-3}r\right)\\
   =&\frac{(n-1)(n-2)}2~\left(|\Sigma|+\omega_{n-1}^{\frac{2}{n-1}}{|\Sigma|}^{\frac{n-3}{n-1}}\right).
\end{align*}
Hence the equality holds in \eqref{maineq} on a geodesic sphere. This completes the proof of Theorem \ref{mainthm}.
\end{proof}

\bibliographystyle{Plain}

\begin{thebibliography}{10}

\bibitem{BC} J. L. M. Barcosa and A. G. Colares, {\it Stability of Hypersurfaces with Constant r-Mean Curvature}, Annals of Global Analysis and Geometry \textbf{15}(1997), 277-297.

\bibitem{Be} W. Beckner, {\it Sharp Sobolev inequalities on the sphere and the Moser-Trudinger inequality,}
Ann. of Math. \textbf{138}(1993), 213-242.

\bibitem{BHW} S. Brendle, P.-K. Hung, and M.-T. Wang, {\it A Minkowski-type inequality for hypersurfaces in the Anti-deSitter-Schwarzschild manifold}, arXiv: 1209.0669.

\bibitem{LimaG} L.L. de lima and F. Gir\~{a}o, {\it An Alexandrov-Fenchel-type inequality in hyperbolic space with an application to a penrose inequality}, arXiv: 1209.0438.

\bibitem{Ding} Q. Ding, {\it The inverse mean curvature flow in rotationally symmetric spaces}, Chinese
Annals of Mathematics, Series B, 1-18 (2010)

\bibitem{GS} E. Gallego and G. Solanes, {\it Integral geometry and geometric inequalities in hyperbolic
space}, Differential Geom. Appl. \textbf{22}(2005), 315-325


\bibitem{Ge} C. Gerhardt, {\it Inverse curvature flows in hyperbolic space}, J. Differential Geom. \textbf{89} (2011), no. 3, 487-527.

\bibitem{GL} P. Guan and J. Li, {\it The quermassintegral inequalities for k-convex starshaped domains,}Adv. Math. \textbf{221}(2009), 1725-1732.

\bibitem{HLP} G.H. Hardy, J.E. Littlewood, G. Polya, {\it Inequalities}, Cambridge Univ. Press, Cambridge, 1934.

\bibitem{Hu} G. Huisken, {\it Flow by mean curvature of convex surfaces into spheres}, J. Differential Geom. \textbf{20} (1984) 237-266.

\bibitem{Hu1} G. Huisken, in preparation

\bibitem{Hu2} G. Huisken and C. Sinestrari, {\it Convexity estimates for mean curvature flow and singularities of mean convex surfaces,}
Acta Math. \textbf{183} (1999) 45-70.


\bibitem{Re} R. Reilly, {\it On the Hessian of a function and the curvatures of its graph}, Michigan Math. J. \textbf{20} (1973) 373-383.

\bibitem{W} M.-T. Wang, {\it Quasilocal mass and surface Hamiltonian in spacetime}, arXiv:1211.1407

\end{thebibliography}

\end{document}